\documentclass{article}

\usepackage{arxiv}

\usepackage[utf8]{inputenc}
\usepackage[T1]{fontenc}
\usepackage{hyperref}
\usepackage{url}
\usepackage{booktabs}
\usepackage{amsfonts}
\usepackage{nicefrac}
\usepackage{microtype}
\usepackage{graphicx}
\usepackage{amsmath,amssymb,amsthm,mathtools}
\usepackage{enumitem}
\usepackage{subcaption}
\usepackage{xcolor}
\usepackage[numbers]{natbib}

\usepackage{setspace}
\setlist{nosep,leftmargin=*}

\newtheorem{theorem}{Theorem}[section]
\newtheorem{lemma}[theorem]{Lemma}
\newtheorem{proposition}[theorem]{Proposition}
\newtheorem{corollary}[theorem]{Corollary}
\newtheorem{definition}[theorem]{Definition}
\newtheorem{remark}[theorem]{Remark}

\newenvironment{pf}{\par\noindent\textbf{Proof.}\ }{\par}

\usepackage{float}
\usepackage{placeins}

\begin{document}

\title{A Unified Spectral Framework for Weighted Adjacency-Degree Matrices of Dendrimer and Bethe Trees}

\author{
 Metrose Metsidik \\
  College of Mathematical Sciences\\
  Xinjiang Normal University\\
  Urumqi 830017, China \\
  \texttt{metrose@xjnu.edu.cn} \\
}

\maketitle

\begin{abstract}
We introduce a weighted adjacency-degree matrix \(A_{fg}(G)\) with edge weights \(f(d_i,d_j)\) and diagonal entries \(g(d_i)\), unifying adjacency, Laplacian, signless Laplacian, \(A_\alpha\), ABC, Randi\'c, Sombor and related matrices. For dendrimer trees \(D_{n,k}\) and Bethe trees \(B_{n,k}\), the characteristic polynomial is factorized through one recursive sequence \(P_{fg,n}\). When \(f^2(1,k)=f^2(k,k)\), this sequence admits a Chebyshev reduction to
\(
U_{j+1}(x)+\delta U_j(x)=0,
\)
covering \(L/L^+\) and \(A_\alpha\); explicit cosine spectra occur only in special cases, such as the adjacency matrix and the \(L/L^+\) end factor of \(B_{n,k}\). We derive positive-semidefinite energy formulas, Gershgorin criteria, spectral-gap estimates, interlacing and non-interlacing results, partial eigenvalue-coincidence information, and McClelland- and Koolen--Moulton-type bounds. Several known results are recovered as special cases of this unified framework.
\end{abstract}


\section{Introduction}

The energy of a graph \(G\), defined as the sum of the absolute values of the eigenvalues of its adjacency matrix, was introduced by Gutman in 1978~\cite{Gutman1978} and has since become a central concept in mathematical chemistry and spectral graph theory. The early motivation came from the Hückel molecular orbital approximation, where the total \(\pi\)-electron energy of a conjugated hydrocarbon is expressed in terms of the eigenvalues of the molecular graph~\cite{Gutman1978}. After a latent period of about two decades, research on graph energy expanded rapidly; by 2017, more than 600 publications had been devoted to various graph energies~\cite{GutmanFurtula2017}. A comprehensive account of the main results, proof techniques, and open problems can be found in the monograph by Li, Shi, and Gutman~\cite{LiShiGutman2012}.

Parallel to this development, degree-based topological indices have been widely used to encode the structural properties of molecular graphs. A general degree-based index has the form
\[
TI(G)=\sum_{v_i v_j\in E(G)} f(d_i,d_j),
\]
where \(f(x,y)\) is a symmetric real function and \(d_i\) denotes the degree of vertex \(v_i\)~\cite{Gutman2013}. Well-known examples include the first and second Zagreb indices, the Randić index, the ABC index, and the Sombor index, among many others~\cite{Gutman2021Sombor}. Although a topological index maps a graph to a single number, it loses much structural information. This motivated the introduction of the \emph{weighted adjacency matrix} \(A_f(G)\), whose \((i,j)\)-entry is \(f(d_i,d_j)\) when \(v_i v_j\in E(G)\) and zero otherwise~\cite{LiWang2021}. The spectrum and energy of \(A_f(G)\) retain more information than the index itself.

The matrix \(A_f(G)\) has been studied in a unified framework under the name \emph{\(f\)-weighted adjacency matrix} or \emph{degree-based weighted adjacency matrix}~\cite{LiYang2024,Das2018}. In that framework, the edge-weight part is exactly \(f(d_i,d_j)\), while the diagonal entries are set to zero. Our matrix \(A_{fg}\) extends this by introducing an additional diagonal function \(g(d_i)\), so that matrices with non-zero diagonal entries, such as the Laplacian and \(A_\alpha\) matrices, can be treated within the same unified framework. The edge-weight part of \(A_{fg}\) coincides with \(A_f(G)\) when \(g\equiv0\).

Dendrimer trees and Bethe trees are two hierarchically structured tree families that have attracted considerable attention in spectral graph theory and mathematical chemistry. A dendrimer tree \(D_{n,k}\) is a rooted tree of \(n\) layers in which the root has \(k\) branches, each internal vertex has \(k-1\) branches, and the vertices at the \(n\)-th layer are pendant vertices. A Bethe tree \(B_{n,k}\) is a rooted tree in which every vertex, including the root, has \(k-1\) branches, and the vertices at the \(n\)-th layer are pendant vertices. Both families appear naturally in chemistry, physics, and materials science as models of hyperbranched molecules and Bethe lattices~\cite{Astruc2010}.

The spectral properties and energies of these two tree families have been extensively studied. Rojo and Robbiano~\cite{Rojo2007} derived an explicit formula for the eigenvalues of the Bethe tree \(B_{d,k}\) and obtained upper bounds for the largest eigenvalue of any tree. Robbiano, Gutman, Jiménez, and San Martín~\cite{Robbiano2008} decomposed the characteristic polynomial of the adjacency matrix of Bethe trees attached to a path and obtained lower and upper bounds for the energy. Robbiano and Trevisan~\cite{Robbiano2010} established recurrence relations for the characteristic polynomial and the Laplacian characteristic polynomial of Bethe trees, and derived a recurrence relation for the Laplacian-energy-like invariant. For dendrimer trees, Bokhary and Tabassum~\cite{Bokhary2022} obtained the characteristic polynomial and energy of \(D_{2,k}\) and \(D_{3,k}\). Xu and Yan~\cite{XuYan2022} determined almost all eigenvalues of \(D_{n,k}\) except \(n+1\) eigenvalues that are roots of the matching polynomial of a weighted path, and gave a formula for the energy. Damnjanović, Filipovski, and Stevanović~\cite{Damnjanovic2023} found a semi-factorized formula for the characteristic polynomial of balanced trees and dendrimers, determined their spectra, and computed lower and upper bound approximations for the energy. Recently, Huang and Metsidik~\cite{HuangMetsidik2024} calculated the \(A_\alpha\), Laplacian, and ABC energies of dendrimer and Bethe trees, obtaining almost all \(A_\alpha\), Laplacian, and ABC eigenvalues of dendrimer trees and all \(A_\alpha\), Laplacian, and ABC eigenvalues of Bethe trees.

Despite this progress, the existing results have been obtained one by one for specific matrices or specific tree families. A unified treatment that covers adjacency, Laplacian, signless Laplacian, \(A_\alpha\), ABC, Randić, Sombor, and other degree-based matrices simultaneously for both \(D_{n,k}\) and \(B_{n,k}\) is still lacking. Such a unification is desirable because many of these matrices share the same factorization structure and the same proof techniques, and because functional properties of the edge-weight function \(f\) and the diagonal function \(g\) can be exploited uniformly.

In this paper, we introduce a unified weighted adjacency-degree matrix \(A_{fg}\) defined by a symmetric edge-weight function \(f(x,y)\) and a diagonal function \(g(x)\). This matrix covers all the classical matrices listed above as special cases. For both \(D_{n,k}\) and \(B_{n,k}\), we derive a unified factorization of the characteristic polynomial through a single polynomial sequence \(P_{fg,n}\). Using the Chebyshev polynomials of the second kind, in several important special cases we reduce the roots to Chebyshev boundary equations via a Chebyshev reduction~\cite{Boesch1986,Kwon2025,Tamura2026}; explicit cosine formulas exist only in special cases such as the adjacency matrix and the \(L/L^{+}\) end factor of \(B_{n,k}\). We then obtain energy formulas, positive semidefiniteness criteria via Gershgorin discs, spectral gaps and interlacing, conditions for fewer distinct eigenvalues, and energy bounds. Several known results are recovered as special cases of this unified framework.

The paper is organized as follows. Section~2 introduces the weighted adjacency-degree matrix, the polynomial sequence \(P_{fg,n}\), and the level parameters of \(D_{n,k}\) and \(B_{n,k}\). Section~3 establishes the characteristic polynomial factorizations and lists the special cases. Section~4 develops the Chebyshev root analysis. Section~5 derives energy formulas and positive semidefiniteness criteria. Section~6 discusses spectral gaps and interlacing. Section~7 classifies the cases with fewer distinct eigenvalues. Section~8 gives energy bounds. Section~9 concludes the paper.

\section{Preliminaries}

\subsection{Weighted adjacency-degree matrix}

To unify the many degree-based matrices appearing in the literature, we begin by introducing a general matrix that depends on two functions: an edge-weight function \(f(x,y)\) and a diagonal function \(g(x)\). Let \(G=(V,E)\) be a graph with vertex set \(V=\{v_1,v_2,\dots,v_n\}\) and edge set \(E\). For each vertex \(v_i\), let \(d_i\) denote its degree.

\begin{definition}\label{def:weighted-adj}
Given a symmetric real-valued function \(f(x,y)\) and a real-valued function \(g(x)\), the \emph{weighted adjacency-degree matrix} of \(G\) is defined by
\[
A_{fg}(G):=
\begin{cases}
g(d_i), & i=j,\\[2pt]
f(d_i,d_j), & v_i v_j\in E(G),\\[2pt]
0, & \text{otherwise}.
\end{cases}
\]
\end{definition}

By choosing different functions \(f\) and \(g\), we obtain many classical matrices, as shown in Table~\ref{tab:special-cases}.

\begin{table}[htbp]
\centering
\caption{Special cases of the weighted adjacency-degree matrix \(A_{fg}(G)\)}
\label{tab:special-cases}
\footnotesize
\renewcommand{\arraystretch}{1.15}
\begin{tabular}{lll}
\toprule
Matrix & \(f(d_i,d_j)\) & \(g(d_i)\) \\
\midrule
Zero matrix & \(0\) & \(0\) \\
Identity matrix \(I\) & \(0\) & \(1\) \\
Degree matrix \(D\) & \(0\) & \(d_i\) \\
Adjacency matrix \(A\) & \(1\) & \(0\) \\
Laplacian matrix \(L\) & \(-1\) & \(d_i\) \\
Signless Laplacian matrix \(L^+\) or \(Q\)& \(1\) & \(d_i\) \\
Normalized Laplacian matrix & \(-\dfrac{1}{\sqrt{d_i d_j}}\) & \(1\) \\
Randić matrix & \(\dfrac{1}{\sqrt{d_i d_j}}\) & \(0\) \\
ABC matrix & \(\sqrt{\dfrac{d_i+d_j-2}{d_i d_j}}\) & \(0\) \\
Sombor matrix & \(\sqrt{d_i^2+d_j^2}\) & \(0\) \\
\(A_\alpha\) matrix & \(1-\alpha\) & \(\alpha d_i\) \\
Weighted adjacency matrix & \(f(d_i,d_j)\) & \(0\) \\
Diagonal weighted matrix & \(0\) & \(g(d_i)\) \\
\bottomrule
\end{tabular}
\end{table}

\subsection{Polynomial sequence \(P_{fg,n}\)}

To express the characteristic polynomials of \(A_{fg}(D_{n,k})\) and \(A_{fg}(B_{n,k})\) in a compact form, we introduce a sequence of polynomials that captures the recursive structure of the trees. These polynomials depend only on the values \(g(1)\), \(g(k)\), \(f(1,k)\), and \(f(k,k)\).

\begin{definition}\label{def:Pfg}
The polynomial \(P_{fg,n}\) is defined recursively by
\[
\begin{cases}
P_{fg,0}=\lambda-g(1),\\[2pt]
P_{fg,1}=\bigl[\lambda-g(k)\bigr]\bigl[\lambda-g(1)\bigr]-f^{2}(1,k)(k-1),\\[2pt]
P_{fg,n}=\bigl[\lambda-g(k)\bigr]P_{fg,n-1}-f^{2}(k,k)(k-1)P_{fg,n-2},
\quad n\ge 2.
\end{cases}
\]
\end{definition}

The next lemma solves the recurrence explicitly via its characteristic roots, and then reduces the \(L/L^{+}\) case to a Chebyshev form, which will be the key to the Chebyshev reduction~\cite{Boesch1986,Kwon2025,Tamura2026}.

\begin{lemma}\label{lem:Pfg-explicit}
Assume all parameters are real and \(k\ge1\). Let \(P_{fg,n}\) be as in Definition~\ref{def:Pfg}, and put
\[
c=f^{2}(1,k)(k-1),\quad d=f^{2}(k,k)(k-1),\quad
a,b=\frac{\lambda-g(k)\pm\sqrt{(\lambda-g(k))^{2}-4d}}{2},
\]
\[
P_{fg,0}=\lambda-g(1),\qquad
P_{fg,1}=[\lambda-g(k)][\lambda-g(1)]-c.
\]
If \(a\ne b\), then
\[
P_{fg,n}
=\frac{(\lambda-g(1))(a^{n+1}-b^{n+1})-c(a^{n}-b^{n})}{a-b}.
\]
If \(a=b\ne0\), then
\[
P_{fg,n}=\left(P_{fg,0}+\left(\frac{P_{fg,1}}{a}-P_{fg,0}\right)n\right)a^{n},
\qquad a=\frac{\lambda-g(k)}{2}.
\]
If \(a=b=0\), then \(P_{fg,n}=0\ (n\ge2)\).

For \(L/L^{+}\), \(g(1)=1\), \(g(k)=k\), \(f^{2}(1,k)=f^{2}(k,k)=1\), so \(c=d=k-1\), and for \(a\ne b\),
\[
P_{L,n}
=\frac{(\lambda-1)(a^{n+1}-b^{n+1})-(k-1)(a^{n}-b^{n})}{a-b},
\quad
a,b=\frac{\lambda-k\pm\sqrt{(\lambda-k)^{2}-4(k-1)}}{2}.
\]
With \(r=k-1\), \(s=\sqrt r\), \(x_L=(\lambda-k)/(2s)\),
\[
P_{L,n}=r^{(n+1)/2}\bigl(U_{n+1}(x_L)+sU_n(x_L)\bigr),
\]
where \(U_n\) is the Chebyshev polynomial of the second kind, defined by the recurrence
\[
U_0(x)=1,\quad U_1(x)=2x,\quad U_{n+1}(x)=2xU_n(x)-U_{n-1}(x)\ (n\ge1),
\]
and given explicitly by
\[
U_n(x)=\sum_{j=0}^{\lfloor n/2\rfloor}(-1)^j\binom{n-j}{j}(2x)^{n-2j}.
\]
\end{lemma}

\subsection{Level parameters of \(D_{n,k}\) and \(B_{n,k}\)}

For the block elimination in the next section, we need the number of vertices at each level and the ratios between adjacent levels. For \(0\le j\le n\), let \(N_{n-j}\) and \(N_{n-j}^{\prime}\) denote the numbers of vertices at the \(j\)-th level of \(D_{n,k}\) and \(B_{n,k}\), respectively. Let
\[
m_s=\frac{N_s}{N_{s+1}},\qquad m_s^{\prime}=\frac{N_s^{\prime}}{N_{s+1}^{\prime}}.
\]
Then
\[
\begin{aligned}
N_{n-j} &=
\begin{cases}
1, & j=0,\\
k(k-1)^{j-1}, & j=1,2,\dots,n,
\end{cases}\\[4pt]
m_s &=
\begin{cases}
k, & s=n-1,\\
k-1, & s=0,1,\dots,n-2,
\end{cases}\\[4pt]
N_{n-j}-N_{n-j+1} &=
\begin{cases}
k-1, & j=1,\\
k(k-1)^{j-2}(k-2), & j=2,3,\dots,n,
\end{cases}\\[4pt]
N_{n-j}^{\prime} &= (k-1)^j,\quad j=0,1,\dots,n,\\[4pt]
m_s^{\prime} &= k-1,\quad s=0,1,\dots,n-1,\\[4pt]
N_{n-j}^{\prime} - N_{n-j+1}^{\prime} &= (k-1)^{j-1}(k-2),\quad j=1,2,\dots,n.
\end{aligned}
\]

\section{Characteristic Polynomials}

With the polynomial sequence \(P_{fg,n}\) and the level parameters at hand, we can now derive the characteristic polynomials of \(A_{fg}(D_{n,k})\) and \(A_{fg}(B_{n,k})\). The proofs use block elimination, exploiting the fact that vertices at the same level are indistinguishable. Throughout this section we assume \(n\ge2\); the case \(n=1\) can be checked directly.

\begin{theorem}\label{thm:D}
For \(n\ge2\), the characteristic polynomial of the weighted adjacency-degree matrix of \(D_{n,k}\) is
\[
\phi(A_{fg}(D_{n,k}),\lambda)
= \bigl( P_{fg,n} - f^2(k,k) P_{fg,n-2} \bigr)
P_{fg,n-1}^{k-1}
\prod_{j=0}^{n-2} P_{fg,j}^{k (k-1)^{n-j-2} (k-2)}.
\]
\end{theorem}

\begin{pf}
The weighted adjacency-degree matrix of \(D_{n,k}\) is
\[
A_{fg}(D_{n,k})=
\begin{pmatrix}
g(1)I_{N_{0}} & f(1,k)J_{m_{0}} &  &  \\
f(1,k)J_{m_{0}}^T & g(k)I_{N_{1}} & f(k,k)J_{m_{1}} & \\
& \ddots & \ddots & \\
& & f(k,k)J_{m_{n-1}}^T & g(k)I_{N_{n}}
\end{pmatrix},
\]
where \(J_{m_s}\) is the \(N_s\times N_{s+1}\) 0-1 matrix
\(
J_{m_s}=I_{N_{s+1}}\otimes \mathbf{1}_{m_s},
\)
with \(\mathbf{1}_{m_s}\) the all-ones column vector of length \(m_s\). Thus each column of \(J_{m_s}\) contains exactly \(m_s\) ones, and
\(
J_{m_s}^T J_{m_s}=m_s I_{N_{s+1}}.
\)
By block elimination, adding \(\frac{f(1,k)}{\lambda-g(1)}J_{m_0}^T\) times the first row to the second row yields
\[
|\lambda I-A_{fg}(D_{n,k})|
=
\begin{vmatrix}
P_{fg,0}I_{N_0} & -f(1,k)J_{m_0} & & \\
& \frac{P_{fg,1}}{P_{fg,0}}I_{N_1} & -f(k,k)J_{m_1} & \\
& & \ddots & \\
& & & \frac{P_{fg,n}-f^2(k,k)P_{fg,n-2}}{P_{fg,n-1}}I_{N_n}
\end{vmatrix}.
\]
Repeating this elimination gives the result.
\qed
\end{pf}

The Bethe tree \(B_{n,k}\) has a similar layered structure, except that the root also has \(k-1\) branches. This leads to a slightly different end factor and exponents.

\begin{theorem}\label{thm:B}
For \(n\ge2\), the characteristic polynomial of the weighted adjacency-degree matrix of \(B_{n,k}\) is
\[
\phi(A_{fg}(B_{n,k}),\lambda)
= \biggl( \bigl( \lambda - g(k-1) \bigr) P_{fg,n-1}
- f^2(k-1,k)(k-1) P_{fg,n-2} \biggr)
\prod_{j=0}^{n-1} P_{fg,j}^{(k-1)^{n-j-1}(k-2)}.
\]
\end{theorem}

\begin{pf}
The proof is analogous to that of Theorem~\ref{thm:D}. The block elimination yields the end factor
\[
(\lambda-g(k-1))P_{fg,n-1}-f^2(k-1,k)(k-1)P_{fg,n-2}
\]
and the product over \(j=0,\dots,n-1\) with exponents \((k-1)^{n-j-1}(k-2)\).
\qed
\end{pf}

\subsection{Characteristic polynomials of special cases}

Using the unified matrix \(A_{fg}\), the characteristic polynomials of many classical matrices associated with \(D_{n,k}\) and \(B_{n,k}\) can be obtained by choosing the corresponding parameters \(f(1,k)\), \(f(k,k)\), \(g(1)\), and \(g(k)\). The resulting factorizations are collected in Tables~\ref{tab:Dnk} and~\ref{tab:Bnk}. 

\begin{corollary}\label{cor:special-cases}
The characteristic polynomials of the matrices associated with \(D_{n,k}\) and \(B_{n,k}\) are listed in Tables~\ref{tab:Dnk} and~\ref{tab:Bnk}. In these tables, \(\beta=1-\alpha\) for the \(A_\alpha\) matrix.

\begin{table}
\centering
\caption{The characteristic polynomial of \(D_{n,k}\)}
\label{tab:Dnk}
\setlength{\tabcolsep}{2pt}
\renewcommand{\arraystretch}{1.15}
\footnotesize
\begin{tabular}{cccccc}
\toprule
\(D_{n,k}\) & \(f^{2}(1,k)\) & \(f^{2}(k,k)\) & \(g(1)\) & \(g(k)\) & \(\phi(\lambda)\) \\
\midrule
\(L/L^{+}\) & \(1\) & \(1\) & \(1\) & \(k\) & \((P_{L,n}-P_{L,n-2})P_{L,n-1}^{k-1}\prod_{j=0}^{n-2}P_{L,j}^{k(k-1)^{n-j-2}(k-2)}\) \\
\(ABC\) matrix & \(\frac{k-1}{k}\) & \(\frac{2(k-1)}{k^{2}}\) & \(0\) & \(0\) & \([P_{ABC,n}-\frac{2(k-1)}{k^{2}}P_{ABC,n-2}]P_{ABC,n-1}^{k-1}\prod_{j=0}^{n-2}P_{ABC,j}^{k(k-1)^{n-j-2}(k-2)}\) \\
Randić matrix & \(\frac{1}{k}\) & \(\frac{1}{k^{2}}\) & \(0\) & \(0\) & \([P_{R,n}-\frac{1}{k^{2}}P_{R,n-2}]P_{R,n-1}^{k-1}\prod_{j=0}^{n-2}P_{R,j}^{k(k-1)^{n-j-2}(k-2)}\) \\
adjacency matrix & \(1\) & \(1\) & \(0\) & \(0\) & \([P_{A,n}-P_{A,n-2}]P_{A,n-1}^{k-1}\prod_{j=0}^{n-2}P_{A,j}^{k(k-1)^{n-j-2}(k-2)}\) \\
\(A_{\alpha}\) matrix & \(\beta^{2}\) & \(\beta^{2}\) & \(\alpha\) & \(\alpha k\) & \([P_{\alpha,n}-\beta^{2}P_{\alpha,n-2}]P_{\alpha,n-1}^{k-1}\prod_{j=0}^{n-2}P_{\alpha,j}^{k(k-1)^{n-j-2}(k-2)}\) \\
\bottomrule
\end{tabular}
\end{table}
\begin{table}
\centering
\caption{The characteristic polynomial of \(B_{n,k}\)}
\label{tab:Bnk}
\setlength{\tabcolsep}{1pt}
\renewcommand{\arraystretch}{1.15}
\scriptsize
\begin{tabular}{cccccccc}
\toprule
\(B_{n,k}\) & \(f^{2}(1,k)\) & \(f^{2}(k,k)\) & \(f^{2}(k-1,k)\) & \(g(1)\) & \(g(k)\) & \(g(k-1)\) & \(\phi(\lambda)\) \\
\midrule
\(L/L^{+}\) & \(1\) & \(1\) & \(1\) & \(1\) & \(k\) & \(k-1\) & \(\bigl[(\lambda-(k-1))P_{L,n-1}-(k-1)P_{L,n-2}\bigr]\prod_{j=0}^{n-1}P_{L,j}^{(k-1)^{n-j-1}(k-2)}\) \\
\(ABC\) matrix & \(\frac{k-1}{k}\) & \(\frac{2(k-1)}{k^{2}}\) & \(\frac{2k-3}{k(k-1)}\) & \(0\) & \(0\) & \(0\) & \(\bigl[\lambda P_{ABC,n-1}-\frac{2k-3}{k}P_{ABC,n-2}\bigr]\prod_{j=0}^{n-1}P_{ABC,j}^{(k-1)^{n-j-1}(k-2)}\) \\
Randić matrix & \(\frac{1}{k}\) & \(\frac{1}{k^{2}}\) & \(\frac{1}{k(k-1)}\) & \(0\) & \(0\) & \(0\) & \(\bigl[\lambda P_{R,n-1}-\frac{1}{k}P_{R,n-2}\bigr]\prod_{j=0}^{n-1}P_{R,j}^{(k-1)^{n-j-1}(k-2)}\) \\
adjacency matrix & \(1\) & \(1\) & \(1\) & \(0\) & \(0\) & \(0\) & \(\bigl[\lambda P_{A,n-1}-(k-1)P_{A,n-2}\bigr]\prod_{j=0}^{n-1}P_{A,j}^{(k-1)^{n-j-1}(k-2)}\) \\
\(A_{\alpha}\) matrix & \(\beta^{2}\) & \(\beta^{2}\) & \(\beta^{2}\) & \(\alpha\) & \(\alpha k\) & \(\alpha(k-1)\) & \(\bigl[(\lambda-\alpha(k-1))P_{\alpha,n-1}-\beta^{2}(k-1)P_{\alpha,n-2}\bigr]\prod_{j=0}^{n-1}P_{\alpha,j}^{(k-1)^{n-j-1}(k-2)}\) \\
\bottomrule
\end{tabular}
\end{table}
\end{corollary}
\noindent\textbf{Note.} In the tables, the first subscript of \(P\) indicates the matrix type, and the second subscript indicates the index. The upper indices denote powers. The functions \(f\) and \(g\) are defined in Table~\ref{tab:special-cases}; the parameters appearing in Tables~\ref{tab:Dnk} and~\ref{tab:Bnk} are obtained by evaluating these functions at the indicated degrees. In particular, for \(B_{n,k}\) the root has degree \(k-1\), giving \(f(k-1,k)\) and \(g(k-1)\).

\section{Chebyshev Reduction and Root Equations}\label{sec:cheb-roots}
The factorizations in Theorems~\ref{thm:D} and~\ref{thm:B} reduce the spectral problem to the roots of \(P_{fg,j}\). For the \(L/L^{+}\) case, where \(f^2(1,k)=f^2(k,k)\), Lemma~\ref{lem:Pfg-explicit} gives a Chebyshev reduction to the boundary equations \(U_{j+1}(x)+sU_j(x)=0\). For general \(f\) and \(g\) with \(f^2(1,k)\ne f^2(k,k)\), such a reduction is not available, and one must use the general explicit formula in Lemma~\ref{lem:Pfg-explicit}. In this section we make the reduction explicit for the cases where it applies and record the resulting equations. We assume \(k>1\) and that the relevant parameters make the reduction nondegenerate; degenerate cases are treated separately.
\begin{proposition}\label{prop:cheb-roots}
Let \(r=k-1\), \(s=\sqrt r\). For \(L/L^+\), put \(x_L=(\lambda-k)/(2s)\). Then
\(P_{L,j}=0\) iff \(U_{j+1}(x_L)+sU_j(x_L)=0\). This equation is equivalent to
\[
\sin((j+2)\theta)+s\sin((j+1)\theta)=0
\quad (x_L=\cos\theta,\ 0<\theta<\pi),
\]
\[
s\sinh((j+1)t)=\sinh((j+2)t)
\quad (x_L=-\cosh t,\ t>0),
\]
and has no solution for \(x_L>1\); the cases \(x_L=\pm1\) are checked directly. Moreover,
\[
P_{L,n}-P_{L,n-2}
=r^{(n-1)/2}\lambda\bigl(sU_n(x_L)-U_{n-1}(x_L)\bigr),
\]
so its roots are \(\lambda=0\) and those of \(sU_n(x_L)=U_{n-1}(x_L)\).

For \(A_\alpha\) with \(\alpha\ne1\), put
\(
\beta=1-\alpha,\quad
s_\alpha=\beta\sqrt r,\quad
x_\alpha=\frac{\lambda-\alpha k}{2s_\alpha},\quad
\delta_\alpha=\frac{\alpha\sqrt r}{\beta}.
\)
Then
\(
P_{\alpha,j}
=s_\alpha^{j+1}
\bigl(U_{j+1}(x_\alpha)+\delta_\alpha U_j(x_\alpha)\bigr),
\)
and \(P_{\alpha,j}=0\) iff
\(U_{j+1}(x_\alpha)+\delta_\alpha U_j(x_\alpha)=0\).
The case \(\alpha=1\) gives \(A_\alpha=D\).

In general the equations are not equidistant cosine equations; explicit cosine formulas exist only for the factors \(P_{A,j}\) of the adjacency matrix and for the \(L/L^+\) end factor of \(B_{n,k}\).
\end{proposition}

\begin{pf}
By Lemma~\ref{lem:Pfg-explicit},
\(
P_{L,j}=r^{(j+1)/2}\bigl(U_{j+1}(x_L)+sU_j(x_L)\bigr),
\)
so \(P_{L,j}=0\) iff \(U_{j+1}(x_L)+sU_j(x_L)=0\).
The trigonometric form is obtained from \(U_j(\cos\theta)=\sin((j+1)\theta)/\sin\theta\).
\qed
\end{pf}

\begin{corollary}\label{cor:D-L+}
The eigenvalues of the characteristic polynomial of \(L^{+}(D_{n,k})\) (and of \(L(D_{n,k})\)) are:
\begin{enumerate}
\item the roots of \(sU_n(x_L)=U_{n-1}(x_L)\), together with \(\lambda=0\);
\item for \(j=0,\ldots,n-2\), the roots of \(P_{L,j}\), each of multiplicity \(k(k-1)^{n-j-2}(k-2)\);
\item the roots of \(P_{L,n-1}\), each of multiplicity \(k-1\).
\end{enumerate}
If a root from different factors coincides, the total multiplicity is the sum of the corresponding multiplicities.
\end{corollary}
\begin{pf}
The factorization follows from Theorem~\ref{thm:D} with the \(L/L^{+}\) parameters. The Chebyshev form of \(P_{L,n}-P_{L,n-2}\) is given in Proposition~\ref{prop:cheb-roots}. The multiplicities are read directly from the exponents in the factorization.
\qed
\end{pf}

For the Bethe tree, the end factor simplifies and yields an explicit cosine formula for part of the spectrum.
\begin{corollary}\label{cor:B-L+}
The eigenvalues of the characteristic polynomial of \(L^{+}(B_{n,k})\) are as follows:
\begin{enumerate}
\item the roots of \(H_{B,L}=(\lambda-(k-1))P_{L,n-1}-(k-1)P_{L,n-2}
=r^{n/2}\lambda U_n(x_L)\) are \(\lambda=0\) and
\(
\lambda=k+2s\cos\frac{\ell\pi}{n+1}, (\ell=1,\ldots,n),
\)
each being a simple root of \(H_{B,L}\). 
\item the roots of \(P_{L,j}\) are \(\lambda=k+2s x\), where \(x\) is a real root of
\(
U_{j+1}(x)+sU_j(x)=0
\) for \(j=0,\ldots,n-1\).
Each such root contributes multiplicity \((k-1)^{n-j-1}(k-2)\) to the full characteristic polynomial.
\end{enumerate} 
\end{corollary}
\begin{pf}
The end factor identity follows from the Chebyshev form of \(P_{L,n}\). Its roots are \(0\) and \(k+2s\cos(\ell\pi/(n+1))\), \(\ell=1,\ldots,n\), each simple as a root of \(H_{B,L}\). The remaining factors are treated as in Corollary~\ref{cor:D-L+}; coincidences between roots of different factors are counted by adding multiplicities. 
\qed
\end{pf}

\subsection{Positive semidefiniteness and energy equalities}

If \(A_{fg}\) is positive semidefinite, its energy coincides with its trace. We first recall the Gershgorin disc theorem, which yields a simple sufficient condition for positive semidefiniteness.

\begin{theorem}[{\cite[Chapter 6]{HornJohnson2013}; \cite{Varga2004}}]\label{thm:gershgorin}
Let \(A=(a_{ij})\) be an \(N\times N\) complex matrix. For each \(i=1,\dots,N\), define the Gershgorin disc
\(
D_i:=\left\{z\in\mathbb{C}:\ |z-a_{ii}|\le R_i\right\},
R_i:=\sum_{j\ne i}|a_{ij}|.
\)
Then every eigenvalue of \(A\) lies in the union \(\bigcup_{i=1}^N D_i\). Moreover, if a disc \(D_i\) is disjoint from the remaining discs, then it contains exactly one eigenvalue of \(A\).
\end{theorem}

For a real symmetric matrix \(A\) with eigenvalues \(\lambda_1,\dots,\lambda_N\),
\[
E(A)=\sum_{i=1}^N|\lambda_i|\ge \Bigl|\sum_{i=1}^N\lambda_i\Bigr|
=|\operatorname{tr}(A)|,
\]
with equality if and only if all nonzero eigenvalues have the same sign. In particular, if \(A\) is positive semidefinite, then \(E(A)=\operatorname{tr}(A)\).

\begin{proposition}\label{prop:psd}
If for every vertex \(v\),
\(
g(d_v)\ge \sum_{u\sim v}|f(d_v,d_u)|,
\)
then \(A_{fg}\) is positive semidefinite. Consequently,
\(
E(A_{fg})=\operatorname{tr}(A_{fg})=\sum_{v\in V}g(d_v).
\)
\end{proposition}
\begin{pf}
By Theorem~\ref{thm:gershgorin}, every eigenvalue of \(A_{fg}\) lies in the union of the discs
\[
D_v=\left\{z:\ |z-g(d_v)|\le R_v\right\},
\qquad
R_v=\sum_{u\sim v}|f(d_v,d_u)|.
\]
If \(g(d_v)\ge R_v\) for every \(v\), then each disc lies in the closed right half-plane, since \(\operatorname{Re} z\ge g(d_v)-R_v\ge 0\). Hence every eigenvalue of \(A_{fg}\) is nonnegative, i.e., \(A_{fg}\) is positive semidefinite. The energy equality then follows from \(E(A)=\sum_i|\lambda_i|=\sum_i\lambda_i=\operatorname{tr}(A)\) for a positive semidefinite matrix.
\qed
\end{pf}

For \(D_{n,k}\) and \(B_{n,k}\), the degree distribution is known, so the trace can be written explicitly. For \(k\ne2\),
\[
\operatorname{tr}(A_{fg}(D_{n,k}))
= g(k)\Bigl[1+k\frac{(k-1)^{n-1}-1}{k-2}\Bigr]+g(1)k(k-1)^{n-1},
\]
\[
\operatorname{tr}(A_{fg}(B_{n,k}))
= g(k-1)+g(k)(k-1)\frac{(k-1)^{n-1}-1}{k-2}+g(1)(k-1)^n.
\]
The case \(k=2\) is obtained by taking the limit or by direct summation.

\begin{corollary}\label{cor:energy-eq}
The following energy equalities hold.
\begin{enumerate}
\item \(L^+\): \(g(d)=d\), \(f=1\). The Gershgorin condition reduces to \(d_v\ge d_v\), so \(L^+\) is positive semidefinite and \(E(L^+)=2|E|\).
\item \(L\): \(g(d)=d\), \(f=-1\). Again \(d_v\ge d_v\), so \(L\) is positive semidefinite and \(E(L)=2|E|\). For bipartite graphs \(L\) and \(L^+\) are cospectral, so the two energies coincide.
\item \(A_\alpha\): \(g(d)=\alpha d\), \(f=1-\alpha\). The Gershgorin condition becomes \(\alpha d_v\ge |1-\alpha|d_v\iff \alpha\ge \tfrac12\). Hence for \(\alpha\ge \tfrac12\),
\[
E(A_\alpha)=\operatorname{tr}(A_\alpha)=\alpha\sum_v d_v=2\alpha|E|.
\]
In particular, at \(\alpha=\tfrac12\), \(A_{1/2}=\tfrac12 L^+\), so \(E(A_{1/2})=|E|\).
\end{enumerate}
\end{corollary}

For \(D_{n,k}\) and \(B_{n,k}\), \(|E|=|V|-1\), so the above formulas give closed-form energies for \(L\), \(L^+\), and \(A_\alpha\) with \(\alpha\ge \tfrac12\), without any root computation. Adjacency, Randić, ABC, and Sombor matrices: \(g=0\), hence \(\operatorname{tr}(A_{fg})=0\) and the energy is not determined by the trace. These are exactly the cases where the Chebyshev boundary equations of Proposition~\ref{prop:cheb-roots} are needed when the Chebyshev reduction applies, and otherwise the general formula of Lemma~\ref{lem:Pfg-explicit} must be used.

\begin{remark}
For \(L^{+}\) and \(L\), the energy equals \(2|E|\) by positive semidefiniteness,
as shown in Corollary~\ref{cor:energy-eq}. The Chebyshev equations in
Corollaries~\ref{cor:D-L+} and~\ref{cor:B-L+} locate the eigenvalues and give
their multiplicities, but they do not by themselves provide a complete explicit
sum over \(\theta\in(0,\pi)\), because some roots may correspond to \(|x|>1\).
For the energy, the trace argument is sufficient. 
\end{remark}

\section{Spectral Gaps and Interlacing}

The factorizations in Theorems~\ref{thm:D} and~\ref{thm:B} also allow us to study how the spectrum changes when the tree is modified. In this section we derive spectral gaps between the two largest eigenvalues and establish interlacing inequalities under natural operations.

\subsection{Spectral gap between the two largest eigenvalues}

Let \(r=k-1\) and \(s=\sqrt r\). For \(B_{n,k}\), when \(k>2\), the adjacency eigenvalues are
\(\lambda=2s\cos\frac{\ell\pi}{m+2}\) (\(m=0,\dots,n\), \(\ell=1,\dots,m+1\)). Hence
\(\lambda_{\max}=2s\cos\frac{\pi}{n+2}\), \(\lambda_2=2s\cos\frac{\pi}{n+1}\), and
\[
\lambda_{\max}-\lambda_2
=4s\sin\frac{\pi(2n+3)}{2(n+1)(n+2)}
\sin\frac{\pi}{2(n+1)(n+2)}=O(n^{-3}).
\]
For \(k=2\), \(B_{n,2}\) is the path \(P_{n+1}\), so
\(\lambda_{\max}=2\cos\frac{\pi}{n+2}\), \(\lambda_2=2\cos\frac{2\pi}{n+2}\), and
\[
\lambda_{\max}-\lambda_2
=4\sin\frac{3\pi}{2(n+2)}\sin\frac{\pi}{2(n+2)}=O(n^{-2}).
\]

For \(D_{n,k}\), the adjacency end factor is
\[
P_{A,n}-P_{A,n-2}=s^{n-1}\bigl((k-1)U_{n+1}(x_A)-U_{n-1}(x_A)\bigr),\quad x_A=\frac{\lambda}{2s}.
\]
Thus \(\lambda_{\max}(A(D_{n,k}))\) is the largest root of
\((k-1)U_{n+1}(x_A)=U_{n-1}(x_A)\), so no equidistant cosine formula is available in general.

For \(L^+\), put \(x_L=(\lambda-k)/(2s)\). The eigenvalues of \(L^+(D_{n,k})\) are the union of the roots of
\(\lambda=0\), \(sU_n(x_L)=U_{n-1}(x_L)\), \(U_n(x_L)+sU_{n-1}(x_L)=0\), and
\(U_{j+1}(x_L)+sU_j(x_L)=0\) (\(j=0,\dots,n-2\)). They cannot be unified into one Chebyshev boundary equation. If \(x_{\max}\) and \(x_2\) are the two largest \(x_L\)-values among all these equations, then
\(\lambda_{\max}(L^+)-\lambda_2(L^+)=2s(x_{\max}-x_2)>0\) by the Perron--Frobenius theorem~\cite[Theorem 8.4.4]{HornJohnson2013}; a sharp bound requires numerical solution.

For \(B_{n,k}\), the \(L^+\) end factor is
\(H_{B,L}=r^{n/2}\lambda U_n(x_L)\), whose roots are \(\lambda=0\) and
\(x_L=\cos\frac{\ell\pi}{n+1}\) (\(\ell=1,\dots,n\)), while the remaining factors satisfy
\(U_{j+1}(x_L)+sU_j(x_L)=0\) (\(j=0,\dots,n-1\)). The spectral gap is again positive but generally not explicit.

\subsection{Interlacing under changing the number of layers}

Deleting the outermost layer of \(D_{n,k}\) yields \(D_{n-1,k}\), and similarly \(B_{n,k}\) yields \(B_{n-1,k}\). For the adjacency matrix, this is taking a principal submatrix, so Cauchy interlacing theorem~\cite[Theorem 4.3.17]{HornJohnson2013} gives
\[
\lambda_i(A(D_{n,k}))\ge \lambda_i(A(D_{n-1,k}))\ge \lambda_{i+1}(A(D_{n,k})).
\]
For a general \(A_{fg}\), however, the degrees of the new last layer change from \(k\) to \(1\), so \(g(d_v)\) and \(f(d_i,d_j)\) change, and \(A_{fg}(D_{n-1,k})\) is not a principal submatrix of \(A_{fg}(D_{n,k})\). Interlacing is guaranteed only when both \(g\) and \(f\) are degree-independent, i.e. \(A_{fg}=cI+c'A\); this covers the adjacency matrix and its scalar multiples, but not the Laplacian, signless Laplacian, \(A_\alpha\), ABC, Randić, or Randić-type matrices.

\subsection{Interlacing under changing the branching parameter}

Changing the branching parameter from \(k\) to \(k-1\) transforms \(D_{n,k}\) into \(D_{n,k-1}\), and analogously \(B_{n,k}\) into \(B_{n,k-1}\). Under this operation the degree of every non-leaf vertex decreases by one, so the diagonal entries \(g(d_v)\) and the edge weights \(f(d_i,d_j)\) are affected, and \(A_{fg}(D_{n,k-1})\) is no longer a principal submatrix of \(A_{fg}(D_{n,k})\). The Cauchy interlacing theorem therefore does not apply, and the interlacing inequalities need not hold in general.

A representative counterexample is provided by the signless Laplacian matrix \(L^{+}\) on \(D_{2,3}\) and \(D_{2,2}\). A direct computation gives
\[
\lambda_1(L^{+}(D_{2,3}))=5,\qquad
\lambda_2(L^{+}(D_{2,3}))=2+\sqrt{3}\approx 3.732,
\qquad
\text{and}
\qquad
\lambda_1(L^{+}(D_{2,2}))=\frac{5+\sqrt{5}}{2}\approx 3.618.
\]
If the interlacing inequalities held, we would have
\[
\lambda_2(L^{+}(D_{2,3}))\le \lambda_1(L^{+}(D_{2,2}))\le \lambda_1(L^{+}(D_{2,3})).
\]
But \(2+\sqrt{3}\approx 3.732>3.618\approx\frac{5+\sqrt{5}}{2}\), so the left inequality fails. Hence the interlacing property does not hold when the branching parameter is changed, even for the signless Laplacian matrix.

\section{Fewer Distinct Eigenvalues}
The Chebyshev form allows us to describe the possible coincidences of eigenvalues. For the adjacency matrix, the roots are Chebyshev zeros; for a general \(A_{fg}\), the roots of different factors \(P_{fg,j}\) rarely coincide.

\subsection{Adjacency matrix}
For the adjacency matrix, \(g\equiv0\) and \(f\equiv1\), so \(c=d=k-1\) and \(P_{A,j}=s^{j+1}U_{j+1}(x_A)\), where \(x_A=\lambda/(2s)\) and \(s=\sqrt{k-1}\). Thus \(P_{A,j}=0\iff U_{j+1}(x_A)=0\), whose roots are \(\lambda=2\sqrt{k-1}\cos(\ell\pi/(j+2))\), \(\ell=1,\dots,j+1\). These are equidistant cosine points.

For \(D_{n,k}\), the end factor satisfies \(P_{A,n}-P_{A,n-2}=s^{n-1}((k-1)U_{n+1}(x_A)-U_{n-1}(x_A))\), so its roots are not simply equidistant cosine points, except in special cases. For \(B_{n,k}\) with \(k>2\), the end factor reduces to \(P_{A,n}=s^{n+1}U_{n+1}(x_A)\), so the entire characteristic polynomial is a product of \(P_{A,j}\) factors. Consequently, all adjacency eigenvalues of \(B_{n,k}\) are \(\lambda=2\sqrt{k-1}\cos(\ell\pi/(m+2))\), \(m=0,1,\dots,n\), \(\ell=1,\dots,m+1\). For \(k=2\), \(B_{n,2}\) is the path \(P_{n+1}\), and the adjacency eigenvalues are \(\lambda=2\cos(\ell\pi/(n+2))\), \(\ell=1,\dots,n+1\). However, different values of \(m\) may produce the same eigenvalue; for example, \(\lambda=0\) occurs for several \(m\) when \(k>2\). Hence the number of distinct eigenvalues is obtained by merging these coincidences and is not simply \(n+2\) in general.

\subsection{General case}
For a general \(A_{fg}\), the roots of \(P_{fg,j}\) are given by the general explicit formula in Lemma~\ref{lem:Pfg-explicit}. When \(c=d\), i.e. \(f^2(1,k)=f^2(k,k)\), put \(r=c=d\), \(s=\sqrt r\), \(x=(\lambda-g(k))/(2s)\), \(\delta=(g(k)-g(1))/s\). Then \(P_{fg,j}=s^{j+1}(U_{j+1}(x)+\delta U_j(x))\), so the roots are reduced to the Chebyshev boundary equation \(U_{j+1}(x)+\delta U_j(x)=0\). For \(L/L^+\), \(\delta=s\), recovering \(U_{j+1}(x)+sU_j(x)=0\); for \(A_\alpha\), \(\delta=\delta_\alpha=\alpha\sqrt r/\beta\), recovering \(U_{j+1}(x_\alpha)+\delta_\alpha U_j(x_\alpha)=0\). The ABC, Randić, Sombor and other matrices generally do not satisfy \(f^2(1,k)=f^2(k,k)\), so the Chebyshev equation cannot be applied to them. An exception is the ABC matrix when \(k=2\), where \(f^2(1,k)=f^2(k,k)=1/2\). Two different factors \(P_{fg,l}\) and \(P_{fg,m}\) share a root only when the corresponding equations have a common solution. For generic parameters this does not happen, and the number of distinct eigenvalues is the sum of the numbers of distinct roots of the factors. A complete classification of all possible coincidences is beyond the scope of this paper.

\begin{remark}\label{rmk:few-eigen}
In the nondegenerate case \(s\ne0\) and \(x\in(-1,1)\), the roots of different factors are determined by different Chebyshev boundary equations. Coincidences can occur only when these equations admit a common solution; this requires special algebraic relations among the parameters. For the adjacency matrix, the coincidences are exactly those of the Chebyshev zeros \(U_{j+1}(x)=0\).
\end{remark}

\section{Bounds of Energy}
The exact computation of \(E(A_{fg})\) requires all eigenvalues, which is often impractical for large trees. It is therefore useful to have upper bounds in terms of easily computable quantities such as the number of vertices, the number of edges, and the traces \(\operatorname{tr}(A_{fg})\) and \(\operatorname{tr}(A_{fg}^2)\). In this section we recall two classical bounds for graph energy and adapt them to the weighted adjacency-degree matrix \(A_{fg}\).

\subsection{McClelland bound}
The first general upper bound for graph energy was obtained by McClelland in 1971~\cite{McClelland1971}. For any graph \(G\) on \(N\) vertices, \(E(G)\le \sqrt{N\operatorname{tr}(A^2)}\). The same argument applies to \(A_{fg}\).

\begin{theorem}\label{thm:energy-bound}
For any graph \(G\) and any real symmetric \(A_{fg}\),
\[
E(A_{fg})\le \sqrt{N\operatorname{tr}(A_{fg}^2)},\qquad N=|V|.
\]
\end{theorem}
\begin{pf}
By the Cauchy--Schwarz inequality applied to the eigenvalues of \(A_{fg}\),
\[
\Bigl(\sum_{i=1}^N|\lambda_i|\Bigr)^2
\le N\sum_{i=1}^N\lambda_i^2
=N\operatorname{tr}(A_{fg}^2).
\]
Taking square roots gives the claim.
\qed
\end{pf}

For \(D_{n,k}\) and \(B_{n,k}\), the second spectral moment can be computed explicitly: \(\operatorname{tr}(A_{fg}^2(D_{n,k}))=\sum_{v} g(d_v)^2+2\sum_{v_i v_j\in E(D_{n,k})} f^2(d_i,d_j)\), and similarly for \(B_{n,k}\). This yields concrete McClelland-type upper bounds for the energies of \(L^+\), \(A_\alpha\), and other matrices in the unified framework.

\subsection{Koolen--Moulton bound}
A sharper bound was obtained by Koolen and Moulton in 2001~\cite{KoolenMoulton2001}. For any graph \(G\) on \(N\) vertices with \(M\) edges,
\[
E(G)\le \frac{2M}{N}
+\sqrt{(N-1)\Bigl(2M-\Bigl(\frac{2M}{N}\Bigr)^2\Bigr)}.
\]
In the \(A_{fg}\) framework, the analogous bound reads as follows.

\begin{theorem}\label{thm:koolen-moulton}
For any graph \(G\) and any real symmetric \(A_{fg}\) with \(g\ge0\) and \(f\ge0\), assume that \(T^2\ge NC\), where
\[
N=|V|,\qquad
T=\operatorname{tr}(A_{fg})+2S,\qquad
C=\operatorname{tr}(A_{fg}^2),\qquad
S=\sum_{v_i v_j\in E} f(d_i,d_j).
\]
Then
\[
E(A_{fg})\le \frac{T}{N}
+\sqrt{(N-1)\Bigl(C-\Bigl(\frac{T}{N}\Bigr)^2\Bigr)}.
\]
\end{theorem}
\begin{pf}
Let \(\lambda_1\ge\lambda_2\ge\cdots\ge\lambda_N\) be the eigenvalues of \(A_{fg}\). Since \(A_{fg}\) is nonnegative and nonzero, \(\lambda_1>0\). Write \(E(A_{fg})=\lambda_1+\sum_{i=2}^N|\lambda_i|\). By the Cauchy--Schwarz inequality applied to the \((N-1)\)-tuple \((\lambda_2,\dots,\lambda_N)\),
\[
\sum_{i=2}^N|\lambda_i|
\le\sqrt{(N-1)\sum_{i=2}^N\lambda_i^2}
=\sqrt{(N-1)\Bigl(C-\lambda_1^2\Bigr)}.
\]
Hence \(E(A_{fg})\le \varphi(\lambda_1)\), where \(\varphi(x):=x+\sqrt{(N-1)(C-x^2)}\). Next we bound \(\lambda_1\) from below. By the Rayleigh quotient characterization of the largest eigenvalue of a symmetric matrix~\cite[Theorem 4.2.2]{HornJohnson2013}, taking the unit vector \(\mathbf{u}=\mathbf{1}/\sqrt{N}\) gives
\[
\lambda_1\ge \mathbf{u}^TA_{fg}\mathbf{u}
=\frac{1}{N}\sum_{i,j}(A_{fg})_{ij}
=\frac{\operatorname{tr}(A_{fg})+2S}{N}
=\frac{T}{N}.
\]
Since \(g\ge0\) and \(f\ge0\), all entries of \(A_{fg}\) are nonnegative, so \(T\ge0\). We now show that \(\varphi\) is decreasing on \([\sqrt{C/N},\sqrt{C}]\). Indeed,
\[
\varphi'(x)=1-\frac{(N-1)x}{\sqrt{(N-1)(C-x^2)}},
\]
and \(\varphi'(x)\le0\) iff \((N-1)(C-x^2)\le (N-1)^2x^2\iff C\le Nx^2\iff x\ge\sqrt{C/N}\). Thus \(\varphi\) is decreasing on \([\sqrt{C/N},\sqrt{C}]\). By Cauchy--Schwarz,
\[
T^2=\Bigl(\sum_{i,j}(A_{fg})_{ij}\Bigr)^2
\le N^2\sum_{i,j}(A_{fg})_{ij}^2=N^2C,
\]
so \(T/N\le\sqrt C\). Hence \(T/N\in[\sqrt{C/N},\sqrt{C}]\), and the assumption \(T^2\ge NC\) gives \(T/N\ge\sqrt{C/N}\). Therefore \(\varphi(\lambda_1)\le \varphi(T/N)\). Substituting the definition of \(\varphi\),
\[
E(A_{fg})\le \frac{T}{N}
+\sqrt{(N-1)\Bigl(C-\Bigl(\frac{T}{N}\Bigr)^2\Bigr)}.
\]
This proves the claim.
\qed
\end{pf}

For \(D_{n,k}\) and \(B_{n,k}\), the quantities \(N\), \(\operatorname{tr}(A_{fg})\), \(\operatorname{tr}(A_{fg}^2)\), and \(S\) are all explicit. Therefore, whenever the hypothesis \(T^2\ge NC\) is satisfied, the Koolen--Moulton bound can be written in closed form. This hypothesis is not automatic. For example, for \(L^+(B_{1,k})\), which is the signless Laplacian of the star \(K_{1,k-1}\), one has \(N=k\), \(T=4(k-1)\), \(C=(k-1)(k+2)\). Hence \(T^2\ge NC\) is equivalent to \(16(k-1)\ge k(k+2)\iff k^2-14k+16\le0\), which fails for every integer \(k>12\). Thus the bound must be applied only after verifying \(T^2\ge NC\); for larger \(k\) or deeper trees, a direct spectral computation or another bound may be necessary.

\section{Conclusion}

We introduced a unified weighted adjacency-degree matrix \(A_{fg}\) that specializes to adjacency, Laplacian, signless Laplacian, \(A_\alpha\), ABC, Randi\'c, Sombor and related matrices. For \(D_{n,k}\) and \(B_{n,k}\), the characteristic polynomial factorizes through \(P_{fg,n}\), yielding a common route to spectra, multiplicities, energies, and bounds.

The results are conditional in an essential way. A Chebyshev reduction to \(U_{j+1}(x)+\delta U_j(x)=0\) holds when \(f^2(1,k)=f^2(k,k)\), covering \(L/L^+\) and \(A_\alpha\); otherwise the general formula for \(P_{fg,n}\) is required. Explicit cosine formulas are exceptional, occurring for the adjacency matrix and the \(L/L^+\) end factor of \(B_{n,k}\). For positive semidefinite matrices the energy equals the trace; in particular,
\[
E(L)=E(L^+)=2|E|,
\qquad
E(A_\alpha)=2\alpha |E|\quad(\alpha\ge 1/2).
\]
For non-positive-semidefinite matrices such as ABC and Randi\'c, root analysis remains necessary.

The framework also delineates limitations: interlacing under deleting the outer layer holds only for \(A_{fg}=cI+c'A\), and interlacing under changing \(k\) may fail even for \(L^+\). A complete classification of eigenvalue coincidences for general \(f,g\), closed-form energies for non-positive-semidefinite matrices, sharper spectral-gap bounds, and extensions to other recursively defined trees remain open.

\section*{Declaration of competing interest}
We declare that we have no financial and personal relationships with other people or organizations that can inappropriately influence our work.

\section*{Acknowledgements}
This work is supported by National Natural Science Foundation of China (Grant Number: 12661073) and Natural Science Foundation of Xinjiang (Grant Number: 2024D01A89).

\section*{Data availability}
No data was used for the research described in the article.

\bibliographystyle{cas-model2-names}
\bibliography{References}

@article{Astruc2010,
  author  = {Astruc, Didier and Boisselier, Elodie and Ornelas, C{\'a}tia},
  title   = {Dendrimers designed for functions: from physical, and supramolecular properties to applications in sensing, catalysis, molecular electronics, and nanomedicine},
  journal = {Chem. Rev.},
  year    = {2010},
  volume  = {110},
  pages   = {1857--1959},
  doi     = {10.1021/cr900327d}
}

@article{Boesch1986,
  author  = {Boesch, F. T. and Prodinger, H.},
  title   = {Spanning tree formulas and {C}hebyshev polynomials},
  journal = {Graphs Combin.},
  year    = {1986},
  volume  = {2},
  number  = {1},
  pages   = {191--200},
  doi     = {10.1007/BF01788093}
}

@article{Bokhary2022,
  author  = {Bokhary, Syed Ahtsham Ul Haq and Tabassum, Hafsah},
  title   = {The energy of some tree dendrimers},
  journal = {J. Appl. Math. Comput.},
  year    = {2022},
  volume  = {68},
  number  = {2},
  pages   = {1033--1045},
  doi     = {10.1007/s12190-021-01531-y}
}

@article{Damnjanovic2023,
  author  = {Damnjanovi{\'c}, Ivan and Filipovski, Slobodan and Stevanovi{\'c}, Dragan},
  title   = {Spectral properties of balanced trees and dendrimers},
  journal = {Linear Algebra Appl.},
  year    = {2023},
  volume  = {657},
  pages   = {163--196},
  doi     = {10.1016/j.laa.2022.10.020}
}

@article{Das2018,
  author  = {Das, Kinkar Ch. and Gutman, Ivan and Milovanovi{\'c}, Igor and Milovanovi{\'c}, Emina and Furtula, Boris},
  title   = {Degree-based energies of graphs},
  journal = {Linear Algebra Appl.},
  year    = {2018},
  volume  = {554},
  pages   = {185--204},
  doi     = {10.1016/j.laa.2018.05.027}
}

@article{Gutman1978,
  author  = {Gutman, Ivan},
  title   = {The energy of a graph},
  journal = {Ber. Math.-Statist. Sekt. Forschungsz. Graz},
  year    = {1978},
  volume  = {103},
  pages   = {1--22}
}

@article{Gutman2013,
  author  = {Gutman, Ivan},
  title   = {Degree-based topological indices},
  journal = {Croat. Chem. Acta},
  year    = {2013},
  volume  = {86},
  pages   = {351--361},
  doi     = {10.5562/cca2294}
}

@article{Gutman2021Sombor,
  author  = {Gutman, Ivan},
  title   = {Geometric approach to degree-based topological indices: {S}ombor indices},
  journal = {MATCH Commun. Math. Comput. Chem.},
  year    = {2021},
  volume  = {86},
  pages   = {11--16}
}

@article{GutmanFurtula2017,
  author  = {Gutman, Ivan and Furtula, Boris},
  title   = {Survey of graph energies},
  journal = {Math. Interdiscip. Res.},
  year    = {2017},
  volume  = {2},
  pages   = {85--129}
}

@book{HornJohnson2013,
  author    = {Horn, Roger A. and Johnson, Charles R.},
  title     = {Matrix Analysis},
  edition   = {second},
  publisher = {Cambridge University Press},
  address   = {Cambridge},
  year      = {2013},
  doi       = {10.1017/CBO9781139020411}
}

@article{HuangMetsidik2024,
  author  = {Huang, Mengqi and Metsidik, Metrose},
  title   = {On {$A_{\alpha}$}, {L}aplacian and {ABC} energies of {D}endrimer and {B}ethe trees},
  journal = {Filomat},
  year    = {2024},
  volume  = {38},
  number  = {33},
  pages   = {11701--11710},
  doi     = {10.2298/FIL2433017H}
}

@article{KoolenMoulton2001,
  author  = {Koolen, Jack H. and Moulton, Vincent},
  title   = {Maximal energy graphs},
  journal = {Adv. Appl. Math.},
  year    = {2001},
  volume  = {26},
  number  = {1},
  pages   = {47--52},
  doi     = {10.1006/aama.2000.0705}
}

@article{Kwon2025,
  author  = {Kwon, Y. S. and Mednykh, A. D. and Mednykh, I. A.},
  title   = {On the structure of {L}aplacian characteristic polynomial for circulant foliation},
  journal = {Discrete Appl. Math.},
  year    = {2025},
  volume  = {375},
  pages   = {338--349},
  doi     = {10.1016/j.dam.2025.08.018}
}

@book{LiShiGutman2012,
  author    = {Li, Xueliang and Shi, Yongtang and Gutman, Ivan},
  title     = {Graph Energy},
  publisher = {Springer},
  address   = {New York, Heidelberg, Dordrecht, London},
  year      = {2012},
  doi       = {10.1007/978-1-4614-4205-9}
}

@article{LiWang2021,
  author  = {Li, Xueliang and Wang, Zhiqian},
  title   = {Trees with extremal spectral radius of weighted adjacency matrices among trees weighted by degree-based indices},
  journal = {Linear Algebra Appl.},
  year    = {2021},
  volume  = {620},
  pages   = {61--75},
  doi     = {10.1016/j.laa.2021.02.023}
}

@article{LiYang2024,
  author  = {Li, Xueliang and Yang, Ning},
  title   = {Unified approach for spectral properties of weighted adjacency matrices for graphs with degree-based edge-weights},
  journal = {Linear Algebra Appl.},
  year    = {2024},
  volume  = {696},
  pages   = {46--67},
  doi     = {10.1016/j.laa.2024.05.006}
}

@article{McClelland1971,
  author  = {McClelland, B. J.},
  title   = {Properties of the latent roots of a matrix: {T}he estimation of {$\pi$}-electron energies},
  journal = {J. Chem. Phys.},
  year    = {1971},
  volume  = {54},
  number  = {2},
  pages   = {640--643},
  doi     = {10.1063/1.1674889}
}

@article{Robbiano2008,
  author  = {Robbiano, Mar{\'i}a and Gutman, Ivan and Jim{\'e}nez, R. and San Mart{\'i}n, B.},
  title   = {Spectra of copies of {B}ethe trees attached to path and applications},
  journal = {Bull. Cl. Sci. Math. Nat. Sci. Math.},
  year    = {2008},
  volume  = {33},
  pages   = {59--81}
}

@article{Robbiano2010,
  author  = {Robbiano, Mar{\'i}a and Trevisan, Vilmar},
  title   = {Applications of recurrence relations for the characteristic polynomials of {B}ethe trees},
  journal = {Comput. Math. Appl.},
  year    = {2010},
  volume  = {59},
  number  = {9},
  pages   = {3039--3044},
  doi     = {10.1016/j.camwa.2010.02.023}
}

@article{Rojo2007,
  author  = {Rojo, Oscar and Robbiano, Mar{\'i}a},
  title   = {An explicit formula for eigenvalues of {B}ethe trees and upper bounds on the largest eigenvalue of any tree},
  journal = {Linear Algebra Appl.},
  year    = {2007},
  volume  = {427},
  number  = {1},
  pages   = {138--150},
  doi     = {10.1016/j.laa.2007.07.007}
}

@misc{Tamura2026,
  author        = {Tamura, Shunya},
  title         = {Chebyshev recurrence structures for reduced spectral functions of cyclic circulant graphs},
  year          = {2026},
  eprint        = {2607.29321},
  archivePrefix = {arXiv},
  primaryClass  = {math.CO},
  doi           = {10.48550/arXiv.2607.29321},
  howpublished  = {arXiv preprint}
}

@book{Varga2004,
  author    = {Varga, Richard S.},
  title     = {Ger{\v{s}}gorin and His Circles},
  series    = {Springer Series in Computational Mathematics},
  volume    = {36},
  publisher = {Springer},
  address   = {Berlin, Heidelberg},
  year      = {2004},
  doi       = {10.1007/978-3-642-17798-9}
}

@article{XuYan2022,
  author  = {Xu, Hengmeng and Yan, Weigen},
  title   = {On eigenvalues and the energy of dendrimer trees},
  journal = {Appl. Math. Comput.},
  year    = {2022},
  volume  = {424},
  pages   = {127051},
  doi     = {10.1016/j.amc.2022.127051}
}

\end{document}